\documentclass[a4paper]{article}
\usepackage[latin1]{inputenc}
\usepackage{amsthm, amssymb,amscd,amsmath,amsfonts,bm}
\usepackage{color}
\usepackage{graphicx}
\usepackage{hyperref}
\hypersetup{colorlinks=true, citecolor=red, urlcolor=blue,linkcolor=blue}

\newtheorem{theorem}{Theorem}
\newtheorem{proposition}{Proposition}
\newtheorem{lemma}{Lemma}
\newtheorem{remark}{Remark}
\numberwithin{proposition}{section}
\numberwithin{lemma}{section}
\numberwithin{remark}{section}
\numberwithin{equation}{section}

\newcommand{\<}{\left\langle}
\renewcommand{\>}{\right\rangle}

\def\N{\mathbb N}
\def\R{\mathbb R}
\def\Pb{\mathbb P}
\def\E{\mathbb{E}\,}
\newcommand{\Const}{\mbox{\rm Const}}
\newcommand{\Cov}{\mbox{\rm Cov}}
\newcommand{\dist}{\mbox{\rm dist}}
\newcommand{\diag}{\mbox{\rm diag}}

\newcommand{\Var}{\mbox{\rm Var}}

\def\bY{{\mathbf Y}}

\def\bB{{\mathbf B}}
\newcommand{\bhY}{\overline{\mathbf{Y}}}
\newcommand{\hY}{\overline{Y}}

\def\CC{{\cal C}}
\def\FF{{\cal F}}
\def\VV{{\cal V}}
\def\balpha{\boldsymbol\alpha}
\def\bmu{\boldsymbol\mu}
\def\bbeta{\boldsymbol\beta}

\def\blambda{\boldsymbol\lambda}
\def\bj{\boldsymbol j}

\def\bhZ{{\mathbf Z}_d}

\title{Central Limit Theorem for the volume of the zero set of
Kostlan-Shub-Smale random polynomial systems.}
\author{Diego Armentano\thanks{Centro de Matem\'{a}tica, 
Universidad de la Rep\'{u}blica, Montevideo, Uruguay. 
E-mail: diego@cmat.edu.uy.}
\;
Jean-Marc Aza\"{i}s\thanks{Institut de Math\'ematiques de Toulouse; 
UMR5219. Universit\'e de Toulouse; CNRS. UT3, F-31062 Toulouse, France. Email: 
jean-marc.azais@math.univ-toulouse.fr
}
\;
Federico Dalmao\thanks{
Departamento de Matem\'{a}tica y Estad\'{i}stica del Litoral, 
Universidad de la Rep\'{u}blica, Salto, Uruguay. 
E-mail: fdalmao@unorte.edu.uy.}
\;
Jos\'e R. Le\'{o}n\thanks{IMERL, Universidad de la Rep\'ublica, 
Montevideo, Uruguay and Escuela de Matem\'atica, Universidad Central de Venezuela. 
E-mail: rlramos@fing.edu.uy.
}}
\begin{document}
\makeatletter
\maketitle
\makeatother
\abstract{We establish the Central Limit Theorem, as the degree goes to infinity, 
for the normalized volume of the zero set of a rectangular 
Kostlan--Shub--Smale random polynomial system. 
This paper is a continuation of 
{\it Central Limit Theorem for the number of real roots of
Kostlan--Shub--Smale random polynomial systems} 
by the same authors 
in which the case of square systems was considered. 
Our main tools are Kac-Rice formula and  an expansion 
of the volume of the level set into the It\^o-Wiener Chaos. 
AMS sujet classification: Primary: 60F05,  30C15. Secondary: 60G60, 65H10}

\section{Introduction}
The problem of studying the number of roots of random algebraic
polynomials has attracted much interest for a long time. It is worth 
mentioning the seminal article of M. Kac on the subject \cite{Kac} 
where a proof of the 
now famous Kac-Rice formula \cite{aw} was given. 
This formula establishes an 
analytical expression for computing the expectation of the number of 
zeros of a Gaussian random process. At the
beginning the main interest was limited to compute the expected value of
such a number.  Later on there were also considered its variance
\cite{Mas} and the Central Limit Theorem (CLT) \cite{Mas2}.

The algebraic systems of random polynomials of several variables were
considered much later motivated by the inspiring work, due to Shub
and Smale, for the understandig of the complexity of B\'ezout's theorem
(see \cite{bez1}, \cite{ss}, \cite{bez3},  \cite{bez4} and
\cite{bez5}). 

Kostlan \cite{Kos} and Shub-Smale \cite{ss} studied random
polynomials systems that are invariant under rotations. 
The properties of invariance of these polynomials 
allow considering them as functions
over the multidimensional sphere. 
In Kostlan's paper 
an explicit expression for the expectation of the number of roots
for a square system of such polynomials was given. 
For rectangular systems the study was
directed to the behavior of the volume of their zero sets. 

Wschebor \cite{ws-var}, in a seminal work, gave for the first time a
bound for a limit variance in the case when the degree of the system is controlled and the
size of the system tends to infinity. 
A central limit theorem for this
asymptotic scheme is still an open problem, even for the particular case of
quadratic systems.

Another asymptotic regime naturally arises, namely: 
to fix the number of equations and variables 
and to let the degree grow to infinity. 
Under this scheme, 
the asymptotic variance for 
the number of roots of square systems 
was obtained in \cite{aadl} and \cite{lt-pu}. 
In the case of rectangular systems the asymptotic variance 
of the volume of the zero level set was given in \cite{lt}. 
Besides, in \cite{aadl2} a central limit theorem 
for the number of roots of the system 
in the square case was obtained.

The present paper 
extends the results listed in the 
last paragraph 
obtaining a central limit theorem 
for the volume of the
zero set of a Kostlan-Shub-Smale random rectangular system 
as the (common) degree tends to infinity 
(see Section 2 for the precise definition). 
We also give an alternative (simpler) 
proof of the limit variance. 

This paper is a continuation of \cite{aadl,aadl2}. 
The proof has the same structure, but new arguments are needed, 
see the comments and remarks.

The fundamental tools for the study of the zero sets are: 
on the one hand the Kac-Rice formulas for calculating 
the mean and variance of the functionals of these sets (\cite{aw})
and on the other hand the well-known 
theorem of the fourth moment to establish the CLT 
for nonlinear functionals of Gaussian processes or fields (\cite{np}).

\section{Main result}
Consider a rectangular system $\bY_d$ of $r$ 
homogeneous polynomial equations in $m+1$ variables 
with common degree $d>1$. 
More precisely, 
let $\bY_d=(Y_1,\dots,Y_r)$ with 
\begin{equation*}
 Y_{\ell}(t)=\sum_{|\bj|= d}a^{(\ell)}_{\bj}t^{\bj};\quad \ell=1,\ldots,r,
\end{equation*}
where 
\begin{enumerate}
  \item $\bj=(j_{0},\dots,j_{m})\in\N^{m+1}$ 
and $|\bj|=\sum^{m}_{k=0}j_{k}$; 
  \item $a^{(\ell)}_{\bj}
  \in\R$, $\ell=1,\dots,r$, $|\bj|= d$;
  \item $t=(t_{0},\dots,t_{m})\in\R^{m+1}$ 
    and $t^{\bj}=\prod^m_{k=0} t^{j_k}_{k}$.
\end{enumerate}

The system $\bY_d$ has the Kostlan--Shub--Smale 
(KSS for short) distribution if the coefficients $a^{(\ell)}_{\bj}$ 
are independent centred normally distributed random variables with variances
\begin{equation*}
  \Var\left(a^{(\ell)}_{\bj}\right)=\binom{d}{\bj}=\frac{d!}{ 
j_0!j_1!\dots j_m!}. 
\end{equation*}

In the case $r<m$, we are interested in the zero level set of $\bY_d$. 
Note that, 
since $\bY_d$ is homogeneous, 
its roots consist of lines through the origin in 
$\R^{m+1}$. Then, 
each root ray of $\bY_d$ in $\R^{m+1}$ corresponds 
exactly to two (opposite) roots of $\bY_d$ on the unit sphere $S^m$ of 
$\R^{m+1}$.  
Hence, the unit sphere $S^m$ is a natural place where to consider the zero set. 

By a Sard-type argument, the zero level set of $\bY_d$ on $S^m$ is, almost
surely, a smooth
sub-manifold of dimension $m-r$ (see for example Aza\"is \& Wschebor
\cite[pp.177]{aw}). We denote by $\VV_{\bY_d}(\mathbf 0)$ the $(m-r)$-volume of the zero level set (on the sphere).

Shub and Smale \cite{ss} and Kostlan \cite{Kos}  proved that
$\E[\VV_{\bY_d}(\mathbf 0)]=2d^{r/2}c_{m,r}$, $r\leq m$, where
$c_{m,r}$
is the geometric measure of the sphere $S^{m-r}$ as a sub-manifold of
$S^m$. 
Letendre \cite{lt} and Letendre-Puchol \cite{lt-pu}, proved 
that there exists $0<V^r_\infty<\infty$ such that
\begin{equation}\label{eq:v-r}
\lim_{d\to\infty}\frac{\Var(\VV_{\bY_d}(\mathbf 0))}{d^{r-m/2}}=V^r_\infty.
\end{equation} 
We include a different proof in Appendix \ref{app}.

We now establish the CLT for the rectangular case.
\begin{theorem}\label{tcl}
Let $\bY_d$ be an $r\times (m+1)$ KSS homogeneous system. 
Then, if $r<m$, 
the standardized $(m-r)$-volume of the zero level set
\[
\bar{\VV}_d=\frac{\VV_{\bY_d}(\mathbf 0)-\E[\VV_{\bY_d}(\mathbf 0)]}{d^{\frac r2-\frac m4}}
\]
converges in distribution as $d\to\infty$ towards a centred normal
random variable with finite positive variance. 
\end{theorem}

\begin{remark}
  Note that the variance of the volume of the zero level set 
  exhibits a surprising behavior as $d\to\infty$. 
  More precisely, as $d\to\infty$, 
  if $r<m/2$, then the variance tends to $0$, 
  if $r=m/2$ it tends to a constant and 
  if $r>m/2$ it tends to infinity. 
  Thus, the normalization in the CLT either reduces or amplifies 
  the oscillations of the volume of the zero level set.
\end{remark}

\begin{remark}
This result can be extended to general functionals of the level sets 
using the same arguments.

Indeed, let us denote the zero level set of $\mathbf Y_d$ as
\[
\mathcal C_{\mathbf Y_d}(0)=\{t\in S^m:\, \mathbf Y_d(t)=0\}.
\]
If $g:S^m\to \R$ is an a.s. continuous function we define the level linear functional 
\[
<g,\mathbf 1_{\mathbf Y_d}>=\int_{\mathcal C_{\mathbf Y_d(0)}}g(t)dt,
\]
where $\mathbf 1_{\mathbf Y_d}$ is the indicator function of
  $\mathcal C_{\mathbf Y_d}(0)$.
Thus $\mathcal V_{\mathbf Y_d}(\mathbf 0) = <\mathbf 1_{\mathbf Y_d},\mathbf 1_{\mathbf Y_d}>$. 
We point out that the study conducted  in the present paper for $\mathcal V_{\mathbf Y_d}(\mathbf 0)$, 
could be made for those functionals. To explain a 
little how to proceed we need in first place a Kac-Rice formula for the first and second order for such a functional. 
This is considered in Chapter 6 
of \cite{aw} (Theorm 6.10 p. 168). In the second place the asymptotic behavior of the variance and the Hermite expansion can be obtained in a 
similar form as we have made here. The reader can consult, for a near matter but for $m=1$, the following preprint \cite{ber}.
\end{remark}

\section{Preliminaries}

This paper is a continuation of \cite{aadl} and \cite{aadl2}. 
For the ease of readability, we gather toghether here 
the notation and some basic results proved in \cite{aadl,aadl2}. 
We also include the celebrated Fourth moment Theorem \cite[Th.11.8.1]{np} 
which is used to get the asymptotic normality.

\subsection{Notation and basic results}
We denote the unit sphere in $\R^{m+1}$ by $S^m$ 
and its volume by $\kappa_m$. 
Concerning integration, 
the variables $s$ and $t$ denote points on $S^m$ and
$ds$ and $dt$ denote the corresponding geometric measure. 
The variables $u$ and $v$ are in $\R^m$ and $du$ and $dv$ are the
associated Lebesgue measure. 
The variables $z$ and $\theta$ are reals and $dz$ and $d\theta$ are the
associated differentials.

We use the Landau's big $O$ and small $o$ notation. 
The set $\N$ of natural numbers contains $0$. 
Besides, $\Const$ will denote a universal constant that may change from a line to another.

\medskip

Lemma 3.1 of \cite{aadl} establishes 
that for an integrable $h:[-1,1]\to\R$ it holds that
\begin{equation}\label{eq:intS}
\int_{S^m\times S^m}h(\<s,t\>)dsdt
=\kappa_m\kappa_{m-1}\int^\pi_0\sin^{m-1}(\theta)h(\cos(\theta))d\theta,
\end{equation}
where $\<\cdot,\cdot\>$ is the usual inner product in $\R^{m+1}$.

Concerning the distribution of $\bY_d$, 
we have
\begin{equation*}
\mathbf \Gamma_d(s,t)
:=\E[{Y_{\ell}(s)Y_{\ell}(t)}]
=\<s,t\>^{d};\quad s,t\in\R^{m+1}.
\end{equation*} 
It follows that 
the distribution of the system $\bY_d$ is invariant under the action of the 
orthogonal group in $\R^{m+1}$. 
For $\ell=1,\ldots,r$, 
we denote by $Y'_\ell(t)$ the derivative (along the sphere) of $Y_\ell(t)$ at the point $t\in S^m$  
and by $Y'_{\ell k}(t)$ its $k$-th component on a given basis of the tangent space of $S^m$ at the point $t$. 
We define the standardized derivative as
\begin{equation}
 \hY_\ell'(t):=\frac{Y_\ell'(t)}{\sqrt{d}},\quad\mbox{and}\quad \bhY'_d(t):=(\hY_1'(t),\ldots,\hY'_r(t)),
\end{equation}
where $\hY_\ell'(t)$ is a row vector. 
In \cite{aadl} it is shown that $(\bY_d(t),\bhY'_d(t))$ is a
vector random field, whose $r(1+m)$ entries are standard normal random variables 
with covariances depending upon the quantities 
\begin{align}\label{eq:deriv}
{\cal A}(\theta)&=-\sqrt{d}\cos^{d-1}(\theta)\sin(\theta),\\
{\cal B}(\theta)&=\cos^{d}(\theta)-(d-1)\cos^{d-2}(\theta)\sin^2(\theta),\notag\\
{\cal C}(\theta)&=\cos^{d}(\theta),\notag\\
{\cal D}(\theta)&=\cos^{d-1}(\theta)\notag.
\end{align}
where $\theta$ is the angle between $s$ and $t$ in $S^m$. 
More precisely, the variance-covariance matrix of the vector  
$$(Y_\ell(s),Y_\ell(t),\hY'_\ell(s),\hY'_\ell(t))$$ 
has the following form
\begin{eqnarray}\label{matrix}
\left[\begin{array}{c|c|c}
A_{11}&A_{12} &A_{13}\\ \hline
A_{12}^{\top}&I_m\,\,&\,A_{23}\\ \hline
A_{13}^{\top}&A_{23}^{\top}\,\,&I_m\\
\end{array}\right],
\end{eqnarray}
where $I_m$ is the $m\times m$ identity matrix, 
$$
A_{11}=\left[\begin{array}{cc}
1&\mathcal{C}\\
\mathcal{C}&1
\end{array}\right], \;
A_{12}= \left[\begin{array}{cccc}
0& 0&\cdots& 0\\
-\mathcal{A}&0& \cdots & 0
\end{array}\right],\:
A_{13}=  \left[\begin{array}{cccc}
\mathcal{A}& 0&\cdots& 0\\
0&0& \cdots & 0
\end{array}\right],$$
and $A_{23}=\diag( \mathcal{B},\mathcal{D},\ldots,\mathcal{D})_{m\times m}.$

Furthermore, the conditional distribution of $(\bhY'_d(s),\bhY'_d(t))$ 
given that $\bY_d(s)=\bY_d(t)=0$ is centered normal 
with variance-covariance matrix
\begin{equation}  \label{e:jm:b}
\left[\begin{array}{c|c}
B_{11}&B_{12} \\ \hline
B_{12}^\top&B_{22} \\
\end{array}\right],
\end{equation}
with $B_{11}=B_{22}= \diag( \sigma^2, 1,\dots,1)_{m\times m}$ and 
$B_{12}= \diag( \sigma^2 \rho, \mathcal{D},\dots,\mathcal{D})_{m\times
m}$. 
Here, 
\begin{equation*}
\sigma^2(\theta)=1-\frac{{\cal A}(\theta)^2}{1-{\cal C}(\theta)^2};\quad
\rho(\theta)=\frac{{\cal B}(\theta)(1-{\cal C}(\theta)^2)-{\cal A}(\theta)^2{\cal C}(\theta)}{1-{\cal C}(\theta)^2-{\cal A}(\theta)^2}.
\end{equation*}

\medskip

Finally, let us retrieve from \cite{aadl} 
the asymptotics and bounds for these quantities 
after scaling $\theta=z/\sqrt{d}$.
\begin{lemma}[\cite{aadl}]\label{l:bounds}
There exists $0<\alpha<\frac12$ such that for $\frac{z}{\sqrt{d}}<\frac{\pi}{2}$ 
it holds that:
\begin{align*}
\left|{\cal A}\right| &\leq z\exp(-\alpha z^2);
&\left|{\cal B}\right| &\leq (1+z^2)\exp(-\alpha z^2);\\
\left|{\cal C}\right| &\leq \left|{\cal D}\right|\leq \exp(-\alpha z^2);
&1-{\cal C}^2 &\geq \Const(1 - \exp(-2\alpha z^2));\\
0\leq 1-\sigma^2 &\leq \Const\cdot\exp(-2\alpha z^2);
&|\rho| &\leq \Const\cdot(1+z^2)^2\exp(-2\alpha z^2). 
\end{align*}
All the functions on the l.h.s. are evaluated at $\theta=z/\sqrt{d}$.
  \hfill $\blacksquare$
\end{lemma}
\begin{lemma}[\cite{aadl}]\label{l:limits}
As 
$d\to +\infty$, it holds that
\[
\cos^{2d}\Big( \frac z{\sqrt d}\Big)\to\exp(-z^2);\quad
\mathcal{A} \to -z  \exp(-z^2/2);
\]
\[
\mathcal{B} \to (1-z^2) \exp(-z^2/2);\quad
\mathcal{C}, \mathcal{D} \to\exp(-z^2/2);
\]
\begin{eqnarray*}
  \sigma^2\Big(\frac z{\sqrt d}\Big) &\to& \frac{1-(1+z^2) \exp(-z^2)}{1- \exp(-z^2)}; \\ 
 \rho  \Big(\frac z{\sqrt d}\Big)  &\to&   \frac{(1 - z^2 - \exp(-z^2))  \exp(-z^2/2) }
   {1-(1+z^2)\exp(-z^2)}. 
\end{eqnarray*}
\hfill $\blacksquare$
\end{lemma}

\subsection{Fourth Moment Theorem} 
We present here the well known Fourth Moment Theorem 
which helps us in the proof of Theorem \ref{tcl}.

Let $\bB=\{B(\lambda):\lambda\geq0\}$ be a standard Brownian motion 
defined on some probability space $(\Omega,\FF,\Pb)$ 
where $\FF$ is the $\sigma$-algebra generated by $\bB$. 
The Wiener chaos is an orthogonal decomposition of $L^2(\bB)= L^2(\Omega,\FF,\Pb)$: 
$$
L^2(\bB)=\bigoplus^\infty_{q=0}\CC_q,
$$ 
where 
$\CC_0=\R$ and for $q\geq 1$, $\CC_q=\{I^{\bB}_q(f_q):f_q\in L^2_s([0,\infty)^q)\}$ 
where $I^{\bB}_q$ is the $q$-folded multiple integral w.r.t. $\bB$ and $L^2_s([0,\infty)^q)$ 
the space of kernels $f_q:[0,\infty)^q\to\R$ which are square integrable and 
symmetric, that is, if $\pi$ is a permutation then 
$f_q(\lambda_1,\dots,\lambda_q)=f_q(\lambda_{\pi(1)},\dots,\lambda_{\pi(q)})$.  
Equivalently, each square integrable functional $F$ of the Brownian motion $\bB$ 
can be written as a sum of orthogonal random variables
\begin{equation*}
F=\E[F]+\sum^\infty_{q=1}I^{\bB}_q(f_{q}),
\end{equation*}
for some uniquely determined kernels $f_q\in L_s^2([0,\infty)^q)$.

Let $f_q,g_q\in L_s^2([0,\infty)^q)$, 
then for  $n=0,\dots,q$ we define the contraction by
\begin{multline}\label{eq:cont}
f_q\otimes_ng_q(\lambda_1,\dots,\lambda_{2q-2n})
=\int_{[0,\infty)^n}f_q(z_1,\dots,z_n,\lambda_1,\dots,\lambda_{q-n})\\
\cdot g_q(z_1,\dots,z_n,\lambda_{q-n+1},\dots,\lambda_{2q-2n})dz_1\dots
  dz_n.\nonumber
\end{multline}

Now, we can state the generalization of the Fourth Moment Theorem.
\begin{theorem}[\cite{pta} Theorem 11.8.3]\label{teo:4o}
Let $F_d$ be in $L_s^2(\bB)$ admit chaotic expansions
\begin{equation*}
F_d=\E[F_d]+\sum^\infty_{q=1}I_q(f_{d,q})
\end{equation*}
for some kernels $f_{d,q}$. 
Then, if $\E[F_d]=0$ and
\begin{enumerate}
  \item for each fixed $q\geq 1$,
    $\Var(I_{q}(f_{d,q}))\xrightarrow[d\to\infty]{}V_q$;

\item $V:=\sum^\infty_{q=1}V_q<\infty$;

\item for each $q\geq 2$ and $n=1,\dots,q-1$, 
$$
\lim_{d\to\infty}\|f_{d,q}\otimes_n f_{d,q}\|_{L_s^2([0,\infty)^{2q-2n})} = 0;
$$

\item $\lim_{Q\to\infty}\limsup_{d\to\infty}\sum^\infty_{q=Q+1}\Var(I_{q}(f_{d,q}))=0$.
\end{enumerate}
Then, $F_d$ converges in distribution towards the $N(0,V)$ distribution. 
 \hfill $\blacksquare$ 
\end{theorem}

\section{Proof of Theorem \ref{tcl}}
Though the main lines of the proof are the same as in \cite{aadl2}, 
there are some important technical differences 
since the volume of the level set is a more complicated object than 
the number of roots in the square case. 
We begin this section with an outline of the proof 
and a description of the main technical differences with \cite{aadl2}.

\begin{itemize}
\item 
First, we obtain the Hermite or Wiener-chaos expansion 
of the standardized $(m-r)$-volume of the zero level set of $\bY_d$ on $S^m$. 

In \cite{aadl2}, the proof of the expansion profited 
of the fact that the number of roots of the system 
is locally constant and of 
B\'ezout's bound for it. 
In the present case, to obtain the chaotic expansion of the volume of the zero level set 
we need to change our arguments to deal with co-area and Rice formulas, 
see Lemma \ref{l:exvol}.
In particular, we need to state the continuity of the volume of the level sets 
with respect to the level. 
this requires a careful use of Gaussian regression.

\item In order to get the limit variance of the $q$-th chaotic component 
we obtain a domination depending on $q$ based on Mehler's formula and in Lemma \ref{l:bounds} above. 
In \cite{aadl} a global domination was obtained 
from the analysis of the limit variance of the number of roots of the system.

The asymptotic normality is obtained in the same way as in \cite{aadl2}. 
In particular, we deduce that the sufficient condition to get 
the asymptotic normality does not depend on the number of equations $r$ 
but on the covariances of the entries of $\bY_d$.

\item We use a convenient partition of the sphere and the existence of a local limit process 
in order to prove the negligibility (of the variance) 
of the tail of the expansion. 
\end{itemize}

\subsection{Hermite expansion of the Volume}
In this part we obtain the Hermite expansion of $\VV_{\bY_d}(\mathbf
0)$; the $(m-r)$-volume of
the zero level set. 

Let
$\delta_\mathbf 0(\mathbf
y)=\prod_{\ell=1}^r\delta_0(y_\ell)$ be the Dirac delta distribution 
for $\mathbf y\in\R^r$.

For $\varepsilon>0$ consider an approximation 
$\frac1{\varepsilon^r}\varphi(\frac{\mathbf y}{\varepsilon})$ 
to Dirac's delta distribution, where we assume that  $\varphi$ is a continuous density function with bounded support. 
Consider also the  function $f$ defined as 
\begin{equation}\label{d:f}
f(\mathbf y')=\sqrt{\det(\mathbf y'(\mathbf y')^{\top})}, 
\end{equation}
where $\mathbf y'=(y'_1;\ldots;y'_r)
= ( y'_{11},y'_{12},\ldots, y'_{1m}; y'_{21},\ldots,y'_{2m};\ldots,\ldots; y'_{r1},\ldots,y'_{rm})$ 
is an $m\times r$ matrix and $(\mathbf{y'})^\top$ its transpose. 
Sometimes, according to convenience, we understand $\mathbf y'$ as a vector in $\R^{r\times m}$.  
Furthermore, for any $\gamma>0$ let 
\begin{equation}\label{d:fgamma}
 f_\gamma(\mathbf y')=f( \gamma y'_{11},y'_{12},\ldots, y'_{1m};\gamma y'_{21},\ldots,y'_{2m};\ldots,\ldots;\gamma y'_{r1},\ldots,y'_{rm}).
\end{equation}
Thus $f=f_1$.

\begin{remark}
 The function $f_\gamma$ plays a key role 
 in the alternative proof of the limit variance 
 of the volume of the zero set of $\bY_d$ 
 since it allows us to deal with the non-homogeneity of $f$ w.r.t. 
 the first column in $\mathbf y$,
 see Appendix.
\end{remark}

Similarly to the case of the number of roots \cite{aadl}, 
we can obtain the following convergence in $L^2(\bB)$.
As said above, the expansion for the volume of the zero level set 
is more subtle than that of the number of roots 
in the square case.

\begin{lemma}\label{l:exvol}
Almost surely and in the $L^2$ sense it holds that
$$
\VV_{\bY_d}(\mathbf 0)=d^{\frac r2}\lim_{\varepsilon\to0}\frac1{\varepsilon^r}
\int_{S^m}\varphi\left(\frac{\mathbf Y_d(t)}{\varepsilon}\right)f(\overline Y'_1( t),\ldots,\overline Y'_r( t))d\mathbf t.
$$ 
\end{lemma}
\begin{proof}
By the co-area formula we have
$$
\frac1{\varepsilon^r}\int_{\R^r}\varphi\left(\frac{\mathbf u}{\varepsilon}\right)
\VV_{\bY_d}(\mathbf u)d\mathbf u=\frac1{\varepsilon^r}\int_{S^m}
\varphi\left(\frac{\mathbf Y_d(t)}{\varepsilon}\right)f(\overline Y'_1(t),\ldots,\overline Y'_r( t))dt,
$$
where $\VV_{\bY_d}(\mathbf u)$ stands for the $(m-r)$-volume of the level set 
  $\{t\in S^m:\, \mathbf Y_d(t)=\mathbf u\}$. 
Changing the variable in the left hand side integral we can write
$$
\mathcal Q_\varepsilon:=\int_{\R^r}\varphi(\mathbf u)\VV_{\bY_d}(\varepsilon\mathbf u)d\mathbf u=\frac1{\varepsilon^r}\int_{  S^m}
\varphi\left(\frac{\mathbf Y_d(t)}{\varepsilon}\right)f(\overline Y'_1( t),\ldots,\overline Y'_r( t)d\mathbf t,
$$
we need to prove  the convergence in $L^2(\bB)$ for this sequence. 
Let us evaluate
\begin{eqnarray}\label{convergencia}\E[(\mathcal Q_\varepsilon-\mathcal V_{\bY_d}(\mathbf 0))^2]=\E[\mathcal Q^2_\varepsilon]-2\E[\mathcal 
Q_\varepsilon\mathcal V_{\bY_d}(\mathbf 0)]+\E[\mathcal V^2_{\bY_d}(\mathbf 0)].\end{eqnarray}
But
$$\E[\mathcal Q^2_\varepsilon]=\int_{\R^r\times\R^r}\varphi(\mathbf u_1)\varphi(\mathbf u_2)
\E[\VV_{\bY_d}(\varepsilon{\mathbf u}_1)\VV_{\bY_d}(\varepsilon{\mathbf u}_2)]d\mathbf u_1d\mathbf u_2$$
and
$$\E[\mathcal Q_\varepsilon\mathcal V_{\bY_d}(\mathbf 0)]
=\int_{\R^r}\varphi(\mathbf u_1)\E[\VV_{\bY_d}(\varepsilon{\mathbf u}_1)\VV_{\bY_d}(\mathbf 0)]d\mathbf u_1.$$
Using the Cauchy-Schwarz inequality we have
\begin{equation}\label{r.h.s.}
\E[\VV_{\bY_d}(\varepsilon{\mathbf u}_1)\VV_{\bY_d}(\varepsilon{\mathbf u}_2)]
\le (\E[\VV_{\bY_d}(\varepsilon{\mathbf u}_1)^2]\E[\VV_{\bY_d}(\varepsilon{\mathbf u}_2)^2])^{\frac12}.
\end{equation}

Furthermore, 
below we show that the right hand side is a continuous function in the variable $\mathbf u$, obtaining
\[
\lim_{\varepsilon\to0}\E[\mathcal Q^2_\varepsilon]\le \E[\mathcal V^2_{\bY_d}(\mathbf 0)].
\]

Moreover, given that the process satisfies the hypothesis of Proposition
  6.12 of Aza\"{\i}s \& Wschebor book \cite{aw}, it holds that $$\mathbb
  P\{\exists t:\, \mbox{rank}(\mathbf Y_d'(t))<r,\,\mathbf Y_d(t)=\mathbf u\}=0.$$ 
Thus by using the implicit function theorem we have that the function $\VV_{\bY_d}(\cdot)$ is a.s. continuous and by a classical result
\[
\mathcal V_{\bY_d}(\mathbf 0)=d^{\frac r2}\lim_{\varepsilon\to0}\frac1{\varepsilon^r}\int_{ S^m}
\varphi\left(\frac{\mathbf Y_d(t)}{\varepsilon}\right)f(\overline Y'_1(t),\ldots,\overline Y'_r(t)dt\mbox{ a.s.}
\]
In this form by Fatou's Lemma we get
$$
  \E[\mathcal V^2_{\bY_d}(\mathbf 0)]\le \lim_{\varepsilon\to0}\E[\mathcal Q^2_\varepsilon]\le \E[\mathcal V^2_{\bY_d}(\mathbf 0)].
$$
  The same result can be obtained for the second addend of (\ref{convergencia}), 
in consequence the convergence in quadratic mean holds.

\medskip

It remains to prove that the right hand side of (\ref{r.h.s.}) is a continuous function, we do that in the following. 
This issue is not present in \cite{aadl2}.
By Kac-Rice formula\\
  \begin{equation}\label{eq:ricevol}
\E[(\mathcal V_{\mathbf Y_d}(\mathbf u))^2]
=d^{r}\int_{ S^m\times S^m}
\E[f(\bhY'_d(t))f(\bhY'_d(s))\mid \bY_d(t)=\bY_d(s)=\mathbf u]\,p_{\bY(t),\bY(s)}(\mathbf u,\mathbf u)dtds.
  \end{equation}
Clearly the density $p_{\bY(t),\bY(s)}(\mathbf u,\mathbf u)$ is continuous as a function of $\mathbf u$. 
We deal now with the conditional expectation.

Recall the notation in \eqref{eq:deriv}. 
Let us define the vector ${\mathbf v}(\<t,s\>)=(\mathcal A,0,\ldots,0)^{\top}$, 
a regression model gives that
\begin{eqnarray*}
\hY'_{\ell}(t)=&\mathbf
  v(\<t,s\>)\frac{\<t,s\>^d}{1-\<t,s\>^{2d}}Y_{\ell}(t)-\mathbf
  v(\<t,s\>)\frac{1}{1-\<t,s\>^{2d}}Y_{\ell}(s)+\xi_{\ell1}(t,s),\\
\hY'_{\ell}(s)=&-\mathbf v(\<t,s\>)\frac{1}{1-\<t,s\>^{2d}}Y_\ell(t)+\mathbf v(\<t,s\>)\frac{\<t,s\>^d}{1-\<t,s\>^{2d}}Y_\ell(s)+\xi_{\ell2}(t,s),
\end{eqnarray*}
with $\xi_{\ell1}(t,s),\xi_{\ell2}(t,s)$ centered Gaussian random variables 
independent from $Y_\ell(s)$ and $Y_\ell(t)$. 
In this form we get that the conditional distribution of 
$\left(\overline{Y}'_{\ell}(t),\overline{Y}'_{\ell}(s)\right)$ 
conditioned to $\mathbf Y(t)=\mathbf Y(s)=\mathbf u$ is 
normal with mean 
\[
\begin{pmatrix}\mathbf v(\<t,s\>)(\frac{\<t,s\>^d-1}{1-\<t,s\>^{2d}})u_{\ell}\\
\mathbf v(\<t,s\>)(\frac{\<t,s\>^d-1}{1-\<t,s\>^{2d}})u_{\ell}\end{pmatrix}
=\begin{pmatrix}-\mathbf v(\<t,s\>)(\frac{1}{1+\<t,s\>^{d}})u_{\ell}\\
-\mathbf v(\<t,s\>)(\frac{1}{1+\<t,s\>^{d}})u_{\ell}\end{pmatrix}
\]
and variance-covariance matrix \eqref{e:jm:b}. 
This result implies that the conditional expectation can be expressed
through the following two vectors:
$$
\zeta_1
:=\left(-\mathbf v(\<t,s\>)\cdot\frac{1}{1+\<t,s\>^{d}}\cdot\frac{u_{1}}\sigma+\overline M_1,\ldots,
-\mathbf v(\<t,s\>)\cdot\frac{1}{1+\<t,s\>^{d}}\cdot\frac{u_{r}}\sigma+\overline M_r\right);$$
$$\zeta_2
:=\left(-\mathbf v(\<t,s\>)\cdot\frac{1}{1+\<t,s\>^{d}}\cdot\frac{u_{1}}\sigma+\overline W_1,\ldots,
-\mathbf v(\<t,s\>)\cdot\frac{1}{1+\<t,s\>^{d}}\cdot\frac{u_{r}}\sigma+\overline W_r\right),$$
where the $(r\times m)$-dimensional vectors
$$(\overline M_1,\ldots,\overline M_r):=(  M_{11},\ldots,M_{1m},  M_{21},\ldots,M_{2m},\ldots, M_{r1},\ldots,M_{rm}),$$
$$( \overline W_1,\ldots,\overline W_r):=( W_{11},\ldots,W_{1m},W_{21},\ldots,W_{2m},\ldots, W_{r1},\ldots,W_{rm}),$$
are such that the $M_{lk}$ (resp. $W_{lk}$) are  independent standard Gaussian random variables and 
\[
\E[M_{l_1k_1}W_{l_2k_2}]=\rho\mathbf 1_{\{l_1=l_2,\,k_1=k_2=1\}}+\mathcal D\mathbf 1_{\{l_1=l_2,\,k_1=k_2>1\}}.
\]
In fact, we have
\begin{equation} \label{e:fz}
\E[f(\bhY'_d(t))f(\bhY'_d(s)]\mid \bY_d(t)=\bY_d(s)=\mathbf u)
=\E[f_{\sigma}(\zeta_1)f_{\sigma}(\zeta_2)].
\end{equation}
Note that this is an example of the importance of the function $f_\sigma$.
By the definition of the function $f_\sigma$ and the form 
of the vectors $\zeta_1$ and $\zeta_2$ this term is evidently a continuous function of the 
variable $\mathbf u.$
\end{proof}

Once we have the approximation in Lemma \ref{l:exvol}, 
the expansion follows as in \cite{aadl}. 

To describe it, we introduce the Hermite polynomials $H_n(x)$, $x\in\R$, by $H_0(x)=1$, $H_1(x)=x$ and $H_{n+1}(x)=xH_n(x)-nH_{n-1}(x)$ for $n\geq 
1$. 
The tensorial versions are defined for multi-indices
$\balpha=(\alpha_1,\ldots,\alpha_r)\in\N^r$ and $\bbeta=(\beta_{11},\ldots,\beta_{1m},\ldots,\beta_{r1},\ldots,\beta_{rm})$, by
\[
\mathbf H_{\mathbf \alpha}(\mathbf y)=H_{\alpha_1}(y_1)\ldots H_{\alpha_r}(y_r),
\]
\[
\widetilde{\mathbf H}_{\bbeta}(\mathbf y')=H_{\beta_{11}}(y'_{11})\ldots H_{\beta_{1m}}(y'_{1m})\ldots H_{\beta_{r1}}(y'_{r1})\ldots 
H_{\beta_{rm}}(y'_{rm}).
\]
We denote the coefficients of Dirac's delta distribution in the
Hermite basis of $L^2(\R^{r},\phi_{r}(\mathbf x)d\mathbf x)$ by
$b_{\balpha}$, where $\phi_{k}$ stands for the standard normal density function in $\R^k$. Readily we can show that 
$b_{\balpha}=0$ if at least one index $\alpha_j$ is odd, otherwise 
\begin{equation*}
b_{\balpha}=\frac1{[\frac{\balpha }2]!}\prod_{j=1}^r\frac1{\sqrt{2\pi}}\bigg[-\frac12\bigg]^{[\frac{\alpha_j}2]}. 
\end{equation*} 

Since $f^2$ is a polynomial, $f\in L^2(\R^{r\times m},\phi_{r\times m}(\mathbf y'))d{\mathbf y'}$. 
For $f$ we have
\begin{equation}\label{e:fbeta}
f(\mathbf y')=\sum_{\bbeta}f_{\bbeta}\widetilde{\mathbf  H}_{\bbeta}(\mathbf y'),
\end{equation}
where $\bbeta$ and $\widetilde{\mathbf H}_{\bbeta}$ are as above and
$$
f_{\bbeta}=\frac1{\bbeta!}\int_{\R^{r\times m}}f(\mathbf y')\widetilde{\mathbf H}_{\bbeta}(\mathbf y')\phi_{r\times m}(\mathbf y')d\mathbf y'.
$$

Let us introduce the functions $$ g_q(\mathbf y,\mathbf
y')=\sum_{|\balpha|+|\bbeta|=q} b_{\balpha}f_{\bbeta} \mathbf
H_{\mathbf\alpha}(\mathbf y)\widetilde{\mathbf H}_{\bbeta}(\mathbf y'),
$$ where $(\mathbf y,\mathbf y')\in\R^r\times\R^{r\times m}$.

Thus similarly to \cite{aadl} we can obtain the expansion. 

\begin{proposition}\label{prop:expansion}
With the same notation as above.  We have, in the $L^2$ sense, that 
$$\bar{\VV}_d=\frac{\VV_{\bY_d}(\mathbf 0)-\E[\VV_{\bY_d}(\mathbf 0)]}{d^{\frac r2-\frac m4}}=
  d^{\frac m4}\sum_{q=1}^\infty\int_{S^{m}}g_q( \mathbf
  Y_d(t),\bhY'(t))\,dt.$$
  \hfill $\blacksquare$
  \begin{remark}\label{VR}
The same type of expansion can be obtained with minor modifications if instead of the volume of the zero set over the whole sphere we consider 
the volume 
of the set restricted to a Borel set $\mathcal G\subset S^m$.
\end{remark}
\end{proposition}

The rest of the proof of Theorem \ref{tcl} 
consists in the verification of the conditions 
of Theorem \ref{teo:4o} with $F_d=\bar{\VV}_d$.

\subsection{Computing the variance of the $q$-th term}
To compute the variance of the $q$-th term in the expansion 
we use Mehler's formula \cite[L. 10.7]{aw}. 

We have

\begin{align*}
  & \E\left[\left(\int_{S^{m}}g_q( \bY_d(t),\bhY'_d(t))dt\right)^2\right]
=\sum_{ |\balpha|+|\bbeta|  =q}\sum_{
  |\balpha'|+|\bbeta'|=q}b_{\balpha}f_{\bbeta}
\,b_{\balpha'}f_{\bbeta'} \\
  &\qquad\qquad\qquad\times\int_{S^m\times S^m}
\E[\mathbf H_{\balpha}(\mathbf Y(t))\mathbf H_{\balpha'}(\mathbf Y(s))\widetilde{\mathbf H}_{\bbeta}(\overline{\mathbf Y}'(t))
\widetilde{\mathbf H}_{\bbeta'}(\overline{\mathbf Y}'(s))]dtds.
\end{align*}

Recall that the coefficients $b_{\balpha}$ are zero if one of the $\alpha_j$ is odd. 
Furthermore, the function
$f(\mathbf y)$ is even with respect to each column, thus its Hermite coefficients 
$$
f_{\bbeta}=f_{\bbeta_1,\bbeta_2,\ldots,\bbeta_r}=\int_{\R^{r\times m}}\sqrt{\det(\mathbf y'(\mathbf y')^{\top})}
\mathbf H_{\bbeta_1}(y'_1)\ldots \mathbf  H_{\bbeta_r}(y'_r)\phi_{r\times m}(\mathbf y')d\mathbf y',
$$
are zero if at least one of the $\bbeta_\ell$ satisfies $|\bbeta_\ell|=2k+1.$ 
In this form  $|\beta_\ell|=\sum_{j=1}^m\beta_{\ell j}$ is necessarily even. 
Moreover,  $q=|\balpha|+|\bbeta|$ is also even.

By independence we have
\begin{align}
  &\E[\mathbf H_{\balpha}(\mathbf Y(t))\widetilde{\mathbf H}_{\bbeta}(\overline{\mathbf Y}(t))
  \mathbf H_{\balpha'}(\mathbf Y(s))
\widetilde{\mathbf H}_{\bbeta'}(\overline{\mathbf Y}'(s))]\nonumber
  \\
  &\qquad\qquad=\prod_{\ell=1}^r\E[H_{\alpha_\ell}(Y_\ell(t))
  \mathbf H_{\bbeta_\ell}(\overline Y'_\ell(t))H_{\alpha'_\ell}(Y_\ell(t))
\mathbf H_{\bbeta'_\ell}(\overline Y'_\ell(t))]\nonumber
  \\
  &\qquad\qquad=\prod_{\ell=1}^r
\E[H_{\alpha_\ell}(Y_\ell(s))H_{\alpha'_\ell}(Y_\ell(t))
H_{\beta_{\ell 1}}(\overline{Y}'_{\ell 1}(s))H_{\beta'_{\ell
  1}}(\overline{Y}'_{\ell 1}(t))]\nonumber\\
  &\qquad\qquad\qquad\qquad\qquad\qquad\qquad\qquad
 \times\prod^m_{j=2}\E[H_{\beta_{\ell j}}(\overline{Y}'_{\ell j}(s))H_{\beta'_{\ell j}}(\overline{Y}'_{\ell 
j}(t))].  \label{e:mehler}
\end{align}
In the second equality we used that the random vectors:
\[
(Y_\ell(s),Y_\ell(t),\overline{Y}'_{\ell1}(s),\overline{Y}'_{\ell1}(t));\quad
(\overline{Y}'_{\ell j}(s),\overline{Y}'_{\ell j}(t));\quad j\geq2
\] 
are independent.  
Using Mehler's formula \cite[L. 10.7]{aw}, 
we get
$$
\E[H_{\beta_{\ell j}}(\overline{Y}'_{\ell j}(s))H_{\beta'_{\ell j}}(\overline{Y}'_{\ell j}(t))]
=\delta_{\beta_{\ell j}\beta'_{\ell j}} \beta_{\ell j}!\,(\rho''_{\ell j})^{\beta_{\ell j}},
$$
where $\rho''_{\ell j}=\rho''_{\ell j}(\<s,t\>)=
\E[\overline{Y}'_{\ell j}(s)\overline{Y}'_{\ell j}(t)]=\<t,s\>^{d-1}.$ 
Since $\sum_{j=1}^m\beta_{\ell j}$ is even, 
we have that either $\beta_{\ell1}$ is even and then $\sum_{j=2}^m\beta_{\ell j}$ is even too 
or $\beta_{\ell1}$ is odd and in this case $\sum_{j=2}^m\beta_{\ell j}$ is also odd. 

For the first factor in the r.h.s. of \eqref{e:mehler}, 
using again Mehler's formula we get 
\[
\E[H_{\alpha_\ell}(Y_\ell(s))H_{\alpha'_\ell}(Y_\ell(t))
H_{\beta_{\ell 1}}(\overline{Y}'_{\ell 1}(s))H_{\beta'_{\ell 1}}(\overline{Y}'_{\ell 1}(t))]=0,
\]
if $\alpha_\ell+\beta_{\ell 1}\neq \alpha'_{\ell}+\beta'_{\ell 1}$.
Otherwise, consider $\Lambda\subset \N^4$ defined by
\[
\Lambda=\{(d_1,d_2,d_3,d_4):d_1+d_2=\alpha_\ell,\,  
d_3+d_4=\beta_{\ell 1},\, d_1+d_3=\alpha'_\ell,\, d_2+d_4=\beta'_{\ell 1}\};
\]
then
\begin{equation*}
\E[H_{\alpha_\ell}(Y_\ell(s))H_{\alpha'_\ell}(Y_\ell(t))
H_{\beta_{\ell 1}}(\overline{Y}'_{\ell 1}(s))H_{\beta'_{\ell 1}}(\overline{Y}'_{\ell 1}(t))]
=\sum_{(d_i)\in\Lambda}\frac{\alpha_\ell ! \alpha'_\ell ! \beta_{\ell 1} ! \beta'_{\ell 1} !}
{d_1!d_2!d_3!d_4!}
\rho^{d_1}
(\rho')^{d_2}
(\rho')^{d_3}
(\rho'')^{d_4},
\end{equation*}
where $\rho=\rho(\<s,t\>)=\E[Y_\ell(s)Y_\ell(t)]$, 
$\rho'=\E[Y_\ell(s)\overline{Y}'_{\ell 1}(t)]
=\E[\overline{Y}'_{\ell 1}(s)Y_\ell(t)]$ 
and 
$\rho''=\E[\overline{Y}'_{\ell 1}(s)\overline{Y}'_{\ell 1}(t)]$. 
Note that the conditions defining the index set $\Lambda$ implies that the first factor in Equation \eqref{e:mehler} is
$$
\prod^r_{\ell=1}
\sum_{(d_i)\in\Lambda}\frac{\alpha_\ell ! \alpha'_\ell ! \beta_{\ell 1} ! \beta'_{\ell 1} !}
{d_1!d_2!d_3!d_4!}
\rho^{d_1}
(\rho')^{d_2+d_3}
(\rho'')^{d_4}.
$$
Hence, if we change $\<s,t\>$ by $-\<s,t\>$, 
for each $\ell$ 
we have the factor 
\begin{eqnarray*}
(-1)^{dd_1}\cdot(-1)^{(d-1)(d_2+d_3)}\cdot(-1)^{dd_4}
=(-1)^{d(d_1+d_4)+(d-1)(d_2+d_3)}\\
=(-1)^{d\alpha_\ell}(-1)^{d\beta'_{\ell_1}}(-1)^{2(\alpha'_\ell-d_1)}=(-1)^{d\beta'_{\ell_1}}
\end{eqnarray*}

\begin{remark}  \label{invariance} 
Changing $\<t,s\>$ by $-\<t,s\>$ in (\ref{e:mehler}) and  considering each term for 
$j=1,\ldots,r$ of the product, either $\beta'_{\ell_1}$ and 
$\sum_{j=2}^m\beta'_{\ell_j}$ are even and then the sign of this term
  does not change, 
or the two numbers are odd and then they have a minus in front 
and the sign neither change. Thus we get that  the complete sign of (\ref{e:mehler}) does not change.
\end{remark}

Let us now define
\begin{equation}\label{H}
  \tilde{\mathcal
  H}_{qd}(\<t,s\>)=\sum_{|(\balpha,\bbeta)|=q}\sum_{|(\balpha',\bbeta')|=q}a_{\bmu}a_{\bmu'}
\E[\mathbf H_{\balpha}(\mathbf Y(t))\widetilde{\mathbf H}_{\bbeta}(\overline{\mathbf Y}(t))\mathbf H_{\balpha'}(\mathbf Y(s))
\widetilde{\mathbf H}_{\bbeta'}(\overline{\mathbf Y}'(s))].
\end{equation}
Set also, for $t\in S^m$, 
\begin{equation}\label{eq:zeta}
  \bhZ(t)=(Z_1(t),\ldots,Z_{r(1+m)}(t))=(\bY_d(t),\bhY'_d(t)). 
\end{equation}
In this manner we can write
\begin{align*}
  &  d^{\frac m2}\E\left[\left(\int_{ S^{m}}g_q(\mathbf Z_d(t))dt\right)^2\right]
=d^{\frac m2}\int_{S^m\times S^m}\tilde{\mathcal H}_{qd}(\<t,s\>)dtds\\
&\qquad\qquad\quad\quad=\kappa_m\kappa_{m-1}d^{m/2}\int_0^{\pi}\sin^{m-1}(\theta)
\tilde{\mathcal H}_{qd}\left(\cos(\theta)\right) d\theta\\
&\qquad\qquad\quad\quad=2\kappa_m\kappa_{m-1}\int_0^{\sqrt d\pi/2}d^{(m-1)/2}
\sin^{m-1}\left(\frac{z}{\sqrt{d}}\right)\tilde{\mathcal H}_{qd}\left(\cos\left(\frac{z}{\sqrt{d}}\right)\right) dz.
\end{align*}
For the second equality we used \eqref{eq:intS} about the integration on the sphere of an 
invariant by rotations function. In the third equality we 
use (deduced from Remark \ref{invariance}) the invariance of the function with respect to the change of variable $\varphi=\frac{\pi}2-\theta$ and 
finally we made $\theta=\frac{z}{\sqrt d}.$

The convergence follows by  dominated convergence using for the covariances 
$\rho_{k,\ell}:=\E[Z_k(s)Z_\ell(t)]$, 
the bounds in Lemma \ref{l:bounds} and the expression for the matrix (\ref{matrix}). 
In this manner the integrand can be bounded by $\Const\, (1+z^2)^q\exp{(-q\alpha z^2)}.$
In conclusion we have
\[
V^r_q:=\lim_{d\to\infty }\displaystyle d^{\frac m2}
\E\left[\left(\int_{ S^{m}}g_q(\mathbf Z_d(t))dt\right)^2\right]=2\kappa_m\kappa_{m-1}\int_0^{\infty}
z^{m-1}\tilde{\mathcal H}_{q}(z) dz\quad q\ge 1,
\]
where 
\[
\tilde{\mathcal H}_{q}(z):=
\lim_{d\to\infty}\tilde{\mathcal H}_{qd}\left(\cos\left(\frac{z}{\sqrt{d}}\right)\right).
\]

Implicit in $\tilde{\mathcal H}_{q}(z)$ 
we use the pointwise limits given in Lemma \ref{l:limits}.

\begin{remark}
In the present case the domination is obtained 
via Mehler's formula 
separately for each fixed $q$. 
In \cite{aadl2} we used a less ad hoc argument 
profiting of the computation of the limit global variance 
(of the number of roots of the system) using Rice formula.
\end{remark}

\subsection{Point 2 of Theorem  \ref{teo:4o}}
Recall from \eqref{eq:v-r} that
$$
V^r_\infty=\lim_{d\to\infty}\Var(\bar{\VV}_d)
=\lim_{d\to\infty}\sum^{\infty}_{q=0}d^{\frac m2}\E\left[\left(\int_{S^{m}}g_q(\mathbf Z_d(t))dt\right)^2\right].
$$ 
The second equality follows from Parseval's identity. 
Thus, by Fatou's Lemma 
$$
V^r:=\sum^{\infty}_{q=0}V^r_q
=\sum^{\infty}_{q=0}\lim_{d\to\infty}d^{\frac m2}\E\left[\left(\int_{S^{m}}g_q(\mathbf Z_d(t))dt\right)^2\right]
\leq V^r_\infty
<\infty.
$$
Actually, equality holds as a consequence of Point 4 and the finiteness of $V^r_\infty$.

\subsection{Normality of the $q$-th term}
Lemma \ref{lemma:cont-cov} below 
gives a sufficient condition on the covariances of the process $\bhZ$
in order to verify the convergence of the norm  of the contractions 
(which in turn gives the asymptotic normality).
Below we write the chaotic components 
\[
I_{q,d}=d^{\frac m4}\int_{S^{m}}g_q(\mathbf Z_d(t))dt.
\]
in Proposition \ref{prop:expansion} 
as multiple stochastic integrals w.r.t. a standard Brownian motion $\bB$ 
and use this fact in order to 
prove Lemma \ref{lemma:cont-cov}.

Let $\bB=\{B(\lambda):\lambda\in[0,\infty)\}$ be a standard Brownian motion on $[0,\infty)$. 
By the isometric property of stochastic integrals 
there exist kernels $h_{t,\ell}$ such that 
the components of the vector $\bhZ$ defined in \eqref{eq:zeta} can be written as:
\begin{equation}\label{eq:h}
Z_\ell(t)=\int^\infty_0h_{t,\ell}(\lambda)dB(\lambda),\ell=1,\dots,r(m+1).
\end{equation}
The kernels $h_{t,\ell}$ can be computed explicitly 
from the definition of $Z_\ell$ 
writing the random coefficients as integrals w.r.t. the Brownian motion
$\bB$. 

The two following lemmas are close to the equivalent lemmas in \cite{aadl},
their proofs are omitted. 
Note that the number of equations $r$ of the system $\bY_d$ 
does not play any role in the condition \eqref{eq:cont-cov} below.

\begin{lemma}\label{lemma:kernels}
With the same notation and assumptions as in Proposition
  \ref{prop:expansion}, $I_{q,d}$ can be written as a multiple stochastic integral
\begin{equation*}
I_{q,d}=I^{\bB}_q(g_{q,d})=\int_{[0,\infty)^q}g_{q,d}({\blambda})dB({\blambda});
\end{equation*}
with 
\begin{equation*}
g_{q,d}(\blambda)=
d^{m/4}\sum_{|\bmu|=q}a_{\bmu}
\int_{S^{m}}
(\otimes^{r(m+1)}_{\ell=1}h^{\otimes \gamma_\ell}_{t,\ell})(\blambda)
dt,
\end{equation*}
where $h_{t,\ell}$ is defined in \eqref{eq:h} 
and $I^{\bB}_q$ 
is the $q$-folded multiple stochastic integral w.r.t. the Brownian
  motion $\bB$.  
\hfill $\blacksquare$
\end{lemma}

As $\Gamma_d(s,t)=\Gamma_d(\<s,t\>)$, 
then $\Gamma_d$ can be seen as a function of one real variable. 
\begin{lemma}\label{lemma:cont-cov}
For $k=0,1,2$, 
  let $\Gamma^{(k)}_d$ indicate the $k$-th derivative of $\Gamma_d:[-1,1]\to\R$. 
If
\begin{equation}\label{eq:cont-cov}
d^{m/3}\int^{\pi/2}_{0}
\sin^{m-1}(\theta)
|\Gamma^{(k)}_d(\cos(\theta))|
  d\theta \mathop{\to}_{d\to +\infty}0,
\end{equation}
then, for $n=1,\dots,q-1$ and $g_{q,d}$ defined in Lemma \ref{lemma:kernels}:
$$
\|g_{q,d}\otimes_n g_{q,d}\|_2\to_{d\to\infty}0.
$$
\hfill $\blacksquare$
\end{lemma}

Therefore, 
it suffices to verify \eqref{eq:cont-cov}. 
For $k=0,1,2$ we have
\begin{align*}
d^{m/3}\int^{\pi/2}_{0}
\sin^{m-1}(\theta)
|\Gamma^{(k)}_d(\cos(\theta))|
d\theta 
  & = d^{m/3}
\int^{\sqrt{d}\pi/2}_{0}
\sin^{m-1}\left(\frac{z}{\sqrt{d}}\right)
\left|\Gamma^{(k)}_d\left(\cos\left(\frac{z}{\sqrt{d}}\right)\right)\right|
\frac{dz}{\sqrt{d}}\\
  & = \frac{1}{d^{m/6}}
\int^{\sqrt{d}\pi/2}_{0}
d^{\frac{m-1}{2}}\sin^{m-1}\left(\frac{z}{\sqrt{d}}\right)
\left|\Gamma^{(k)}_d\left(\cos\left(\frac{z}{\sqrt{d}}\right)\right)\right|dz.
\end{align*}
Now $d^{\frac{m-1}{2}}\sin^{m-1}\left(z/\sqrt{d}\right)\leq z^{m-1}$ 
and taking the worst case in Lemma \ref{l:bounds} 
we have $|\Gamma^{(k)}_d(z/\sqrt{d})|\leq (1+z^2)\exp(-\alpha z^2)$. 
Hence, the last integral is convergent and (\ref{eq:cont-cov}) follows.

\subsection{Point 4 in Theorem \ref{teo:4o}}
Let 
$\pi^Q
$ be the projection 
on $\oplus_{q\geq Q}\CC_q$. 
We need to bound the following quantity uniformly in $d$
\begin{eqnarray}\label{tailvariance}
\frac{d^{m/2}}4\Var(\pi^Q(\VV_{\bY_d}(\mathbf 0))
=\frac14\sum_{q\geq Q} d^{m/2}\int_{S^m\times S^m}{\tilde{\mathcal H}}_{qd}(\<s,t\>)dsdt,
\end{eqnarray}
where ${\tilde{\mathcal H}}_{qd}$ is defined in (\ref{H}).

In order to bound this quantity we split the integral 
depending on the 
(geodesical) distance between $s,t\in S^m$ 
\begin{equation*}
\dist(s,t)=\arccos(\left\langle s,t\right\rangle),
\end{equation*}
into the integrals over the regions
$\{(s,t):\dist(s,t)<a/\sqrt{d}\}$ and its complement, 
$a$ will be chosen later. 
We bound each part in the following two subsections.

\subsubsection{Off-diagonal term}
In this subsection we consider the integral in the r.h.s. of \eqref{tailvariance} 
restricted to the off-diagonal region $\{(s,t):\dist(s,t)\geq a/\sqrt{d}\}$. 
This is the easier case since the covariances of $\bhZ$ are bounded away from $1$. 

We use Arcones' Lemma (\cite{arcones}, page 2245). 
Let $X$ be a standard Gaussian vector on $\R^N$ and $h:\R^N\to\R$ a measurable function 
such that $\mathbb E[h^2(X)]<\infty$ and let us consider its $L^2$ convergent Hermite's expansion 
$$h(x)=\sum_{q=0}^\infty\sum_{|\mathbf k|=q}h_{\mathbf k}\mathbf H_{\mathbf k}(x).$$ 
The Hermite rank of $h$ is defined as 
$$\mbox{rank}(h)=\inf\{\tau:\, \exists\; \mathbf k\,, |\mathbf k|=\tau\,; \E[(h(X)-\E h(X))\mathbf H_{\mathbf k}(X)]\neq0\}.$$
Then, we have 
\begin{lemma}[\cite{arcones}]
Let $W=(W_1,\ldots,W_N)$ and $Q=(Q_1,\ldots,Q_N)$ be two mean-zero Gaussian random vectors on $\R^N$. Assume that
\begin{equation*}
\mathbb E[W_jW_k]=\mathbb E[Q_jQ_k]=\delta_{j,k}, 
\end{equation*}
for each $1\le j,k\le N$. 
We define 
\begin{equation*}
r^{(j,k)}=\mathbb E[W_jQ_k].
\end{equation*}
Let $h$ be a function on $\R^N$ with finite second moment and Hermite rank $\tau,$ $1\le\tau<\infty,$ define
\begin{equation*}
\psi:= \max\left\{\max_{1\le j\le N}\sum_{k=1}^N|r^{(j,k)}|,\max_{1\le k\le N}\sum_{j=1}^N|r^{(j,k)}|\right\}.
\end{equation*}
Then
\begin{equation*}
|\Cov(h(W),h(Q))|\le \psi^{\tau}\mathbb E[h^2(W)].
\end{equation*}
\hfill $\blacksquare$
\end{lemma}
We apply this lemma for $N=r\times(1+m)$, $W=\mathbf Z(s)$, $Q=\mathbf Z(t)$ 
and to the function $h(\mathbf y,\mathbf y')=g_q(\mathbf y,\mathbf y')$. Recalling that 
$\rho_{k,\ell}(s,t)=\rho_{k,\ell}(\<s,t\>)=\E[Z_k(s)Z_\ell(t)]$, 
the Arcones' coefficient is now 
$$\psi(s,t)=\max\left\{\sum_{1\le k\le m+m^2}|\rho_{k,\ell}(s,t)|,\sum_{1\le \ell\le m+m^2}|\rho_{k,\ell}(s,t)|\right\}.$$
Thus
\[
|{\tilde{\mathcal H}}_{qd}(\<s,t\>)|\le \psi(\<s,t\>)^q||g_q||^2,
\]
being 
$\|g_q\|^2=\E[g^2_q(\zeta)]$ for standard normal $\zeta$. 

The following lemma is obtained as in \cite{aadl2}.

\begin{lemma}\label{l:normaG}
For $g_q$  
it holds that 
$ ||g_q||^2\le ||f||^2_2 $.
\end{lemma}

To bound the Arcones' coefficient $\psi(\<s,t\>)$ 
we use the expressions in (\ref{eq:intS}) 
thanks to the invariance of the distribution of $\bY_d$ (and $\bhZ$) under isometries.
It is not hard to see that 
the maximum in the definition of $\psi$ is $|\mathcal C|+\mathcal |\mathcal A|$, see \eqref{eq:deriv}. 
From Lemma \ref{l:bounds} 
it follows that $|\mathcal C|+|\mathcal A|\le e^{-\alpha z^2}(1+z).$  
For $z=2$ the bound takes the value
$2e^{-4\alpha}$ which is less or equal to one if $\alpha\ge \frac14\log2$, this is always possible because the only restriction that we have is 
$\alpha<\frac12.$ 
Moreover, for $\delta$ small enough $e^{-\alpha z^2}(1+z)\ge1$ if $z<\delta.$ 
Thus, there exists an $a<2$ such that for all $z\ge a$ it holds $|\mathcal C|+|\mathcal A|<r_0< 1$. 
Hence,
\begin{multline*}
\sup_{d}\sum_{q\geq Q}\frac{d^{m/2}}4
\int_{\left\{(s,t):\,\dist(s,t)\geq \frac{a}{\sqrt{d}}\right\}}{\tilde{\mathcal H}}_{q,d}(\<s,t\>)dsdt\\
=\sup_{d}\frac{C_m}4\left|\sum_{q\geq Q} d^{\frac{m-1}2}
\int_{a}^{\sqrt d\pi}\sin^{m-1}\left(\frac z{\sqrt d}\right)
{\tilde{\cal H}}^{q}_d\left(\cos\left(\frac z{\sqrt d}\right)\right)dz\right|\\
\le C_m||f||_2^2\sum_{q\geq Q} r^{q-1}_0\int_{a}^\infty z^{m-1}(1+z)e^{-\alpha z^2}dz
\mathop{\to}_{Q\to\infty}0. 
\end{multline*}

\subsubsection{Diagonal term}
It remains to prove that the integral in the r.h.s. of \eqref{tailvariance} 
restricted to the diagonal region $\{(s,t):\,\dist(s,t)< a/\sqrt{d}\}$
tends to $0$ as $Q\to\infty$ uniformly in $d$, 
$a<2$ is fixed. 
This is the difficult part, we use an indirect argument.

Next proposition, whose proof is similar to that in \cite{aadl2}, gives a
convenient partition of the sphere based on the hyperspheric
coordinates.
For $\Theta=(\theta_1,\dots,\theta_{m-1},\theta_m)\in[0,\pi)^{m-1}\times[0,2\pi)$ 
we write $x^{(m)}(\Theta)=(x^{(m)}_1(\Theta),\dots,x^{(m)}_{m+1}(\Theta))\in S^m$ in the following way
\begin{equation*}
x^{(m)}_k(\Theta)  = \prod^{k-1}_{j=1}\sin(\theta_j)\cdot\cos(\theta_k),\;k\leq m\textrm{ and }
x^{(m)}_{m+1}(\Theta)  = \prod^{m}_{j=1} \sin(\theta_j);
\end{equation*}
with the convention that $\prod^0_1=1$. 

Define the hyperspherical rectangle (HSR for short) with 
center $x^{(m)}(\tilde{\theta})$ with $\tilde{\theta}=(\tilde{\theta}_1,\ldots,\tilde{\theta}_m)$ 
and vector radius $\tilde{\eta}=(\tilde{\eta}_1,\ldots,\tilde{\eta}_m)$ as
\begin{equation*}
HSR(\tilde{\theta},\tilde{\eta})=
\{x^{(m)}(\theta):|\theta_i-\tilde{\theta}_i|<\tilde{\eta}_i,i=1,\ldots,m\}.
\end{equation*}
Let $T_tS^m$ be the the tangent space to $S^m$ at $t$. This space  can be identified with $t^\bot\subset\R^{m+1}$.   
Let $\phi_t:S^m\to t^\bot$ be the orthogonal projection over $t^{\bot}$, 
$\mathcal{C}_{\mathbf Y_d}(\mathbf{0})$ be the zero set of $\bY_d$ on
$S^m$ and $\VV$ its volume on $S^m$. 
\begin{proposition}\label{p:partition}
For $d$ large enough, 
there exists a partition of the unit sphere $S^m$ into HSRs $R_j:j=1,\ldots,k(m,d)=O(d^{m/2})$ 
and an extra set $E$ such that
\begin{enumerate}
 \item $\Var(\VV(\mathcal{C}_{\mathbf Y_d}(\mathbf{0})\cap E))=o(d^{r-\frac m2})$.
 \item The HSRs $R_j$ have diameter $O(\frac{1}{\sqrt{d}})$ and if $R_j$ and $R_\ell$ do not share any border point 
 (they are not neighbors), then $\dist(R_j,R_\ell)\geq \frac{1}{\sqrt{d}}$. 
 \item The projection of each of the sets $R_j$ on the tangent space at its center $c_j$, 
after normalizing by the multiplicative factor $\sqrt{d}$, converges to
    the rectangle $ [-1/2, 1/2]^m$ in the sense of Hausdorff distance. 
That is, 
the  Hausdorff distance of 
$$
\left[-\frac12,\frac12\right]^m\setminus\sqrt{d}\;\phi_{c_j}(R_j)
$$
\end{enumerate}
tends to $0$ as $d\to\infty$.
\hfill $\blacksquare$
\end{proposition}

Arguing as in \cite{aadl2} it suffices to bound 
$\Var\left(\pi^Q(\VV(\mathcal{C}_{\mathbf Y_d}(\mathbf{0})\cap R_j))\right)$,  
where $R_0$ is an HSR contained in the spherical cap
\[
C(e_0,\gamma /\sqrt{d})=\{s:d(s,e_0)<\gamma /\sqrt{d}\}.
\]
for some $\gamma$ depending on $m$.


We use the local chart $\phi:C(e_{0},\gamma /\sqrt{d}) \to B(0,\sin(\gamma /\sqrt{d}))\subset \R^m$  
defined by 
\[
\phi^{-1}(u)=(\sqrt{1-\|u\|^2},u),\quad u\in B(0,\sin(\gamma /\sqrt{d})),
\] 
to project this set over the tangent space. 
Define the random field 
$\mathcal Y_d:B\left(0,\gamma \right)\subset\R^m\to\R^r$, as 
\[
\mathcal Y_d(u)=\bY_d(\phi^{-1}(u \sqrt{d})).
\]
Observe that the $\ell$ coordinates, ${\cal Y}^{(\ell)}_d$ say, of ${\cal Y}_d$ are independent. 
Clearly, the zero set of $\bY_d$ on $R\subset C(e_{0},\gamma /\sqrt{d})$ 
and the zero set of ${\mathcal Y}_d$ on $\phi(R \sqrt{d})\subset B(0,\gamma)$ coincide. 
That is
\[
\mathcal{C}_{\mathbf Y_d}(\mathbf{0})\cap R=\mathcal{C}_{\mathcal Y_d}(\mathbf{0})\cap \phi(R \sqrt{d}).
\]

\begin{proposition}\label{p:local}
The sequence of processes ${\mathcal Y}^{(\ell)}_d(u)$ 
as well as its first and second order derivatives 
converge in the finite dimensional distribution sense towards the mean zero Gaussian processes $\mathcal Y_{\infty}$ 
with covariance function $\Gamma(u,v)=e^{-\frac{||u-v||^2}2}$ 
and its corresponding derivatives. \hfill $\blacksquare$
\end{proposition}
The proof of this proposition can be consulted in  \cite{ns} and also in \cite{aadl2}.

\begin{remark}
The local limit process $\mathcal Y_\infty$ has as coordinates
$(\mathcal Y^{(1)}_\infty,\ldots,\mathcal Y^{(r)}_\infty)$ such that each one of them is an
independent copy of the random field with  covariance
$\Gamma(u)=e^{-\frac{||u||^2}2}$, $u\in\R^m$. Then its covariance matrix writes 
$$\tilde \Gamma(u)= \diag(\Gamma(u),\ldots,\Gamma(u)).$$
The   second derivative matrix  $\tilde \Gamma''(u)$ 
can be written in a similar way, but here the blocks are equal to the matrix
$\Gamma''(u)=(a_{ij})$ where $a_{ij}=e^{-\frac{||u||^2}2}H_1(u_i)H_1(u_j)$ if $i\neq j,$ and $a_{ii}=e^{-\frac{||u||^2}2}H_2(u_i).$ 
It follows from \cite[Th.12]{al} that the variance of 
$\mathcal V(\mathcal{C}_{\mathcal Y_\infty}(\mathbf{0})\cap K)$ is finite 
for any compact $K$.
\end{remark}


\medskip 

The proof of the main theorem will be achieved as soon as we have proved 
the following proposition.
\begin{proposition}\label{p:cap}
For 
$\varepsilon>0$ there exist $Q_0$ and $d_0$ such that for $Q\geq Q_0$
\[
\sup_{d>d_0}\E[(\pi^Q(\mathcal V(\mathcal{C}_{\mathbf Y_d}(\mathbf{0})\cap R_j)))^2]<\varepsilon.
\]
\end{proposition}
\begin{proof}
Let $R=R_0\subset C(e_{0},\gamma /\sqrt{d})$, 
By Remark \ref{VR}, 
the Hermite expansion holds true also for the volume of the zero set of $\bY_d$ on any subset of $S^m$. 
Hence, 
\[
\mathcal V(\mathcal{C}_{\mathbf Y_d}(\mathbf{0})\cap R)=\sum_{q=0}^\infty d^{\frac r2}\int_{R}g_q(\bhZ(t))dt.
\] 
Let us define $\tilde R=\phi(R)\subset B(0,\sin\frac a {\sqrt d})\subset \R^m$. 
It follows that
\[
\mathcal V(\mathcal{C}_{\mathcal Y_d}(\mathbf{0})\cap \tilde R)=\mathcal V(\mathcal{C}_{\mathbf Y_d}(\mathbf{0})\cap R)
=\sum_{q=0}^\infty d^{\frac r2}\int_{\tilde R}g_q({\mathcal
  Y}_d(u),{\mathcal Y}_d'(u))J_\phi(u)du,
\] 
where $J_\phi(u)=(1-\|u\|^2)^{-1/2}$ is the Jacobian. Rescaling $u=v/\sqrt{d}$ we get
\[
\frac{\mathcal V(\mathcal{C}_{\mathcal Y_d}(\mathbf{0})\cap \tilde R)}{d^{\frac r2-\frac m4}}
=\sum_{q=0}^\infty\int_{\sqrt d \tilde R}
g_q\left({\mathcal Y}_d\left(\frac v{\sqrt d}\right),{\mathcal
  Y}_d'\left(\frac v{\sqrt d}\right)\right)J_{\phi}\left(\frac{v}{\sqrt{d}}\right)dv.
\]
Besides, Kac-Rice formula, the domination for ${\cal H}_{qd}$ previously  obtained, 
the convergence of ${\mathcal Y}_d$ to ${\mathcal Y}_\infty$ in Proposition \ref{p:local} 
and the convergence, after normalization, of $\bar{R}$ to $[-1/2,1/2]^m$ in Proposition \ref{p:partition} 
yield
\begin{equation}\label{eq:conv-tan}
\Var\left(\frac{\mathcal V(\mathcal{C}_{\mathcal Y_d}(\mathbf{0})\cap \tilde R)}{d^{\frac r2-\frac m4}}\right)
\mathop{\to}_{d\to\infty}
  \Var\left(\mathcal V\left(\mathcal{C}_{\mathcal Y_\infty}(\mathbf{0})\cap \left[-\frac12,\frac12\right]^m\right)\right). 
\end{equation}
In fact, 
for the second moment we have \\
$\displaystyle\E\left[\left(\frac{\mathcal V(\mathcal{C}_{\mathcal Y_d}(\mathbf{0})\cap \tilde R)}{d^{\frac r2-\frac m4}}\right)^2\right]$
$$
=d^m\int_{\tilde R\times\tilde R}\E[f(\mathcal Y'_d(u))f(\mathcal Y'_d(v))\,|{\mathcal Y_d}(u)={\mathcal
  Y}_d(v)=0]\; p_{u,v}(0,0)J_{\phi}(u)J_{\phi}(v)dudv$$

\vspace{0.2cm}

$$
=\int_{\sqrt d \tilde R\times\sqrt d\tilde R}\E\left[f\left(\mathcal Y'_d\Big(\frac{u}{\sqrt d}\Big)\right)
f\left(\mathcal Y'_d\Big(\frac{v}{\sqrt d}\Big)\right)\,|{\mathcal Y}_d\Big(\frac{u}{\sqrt d}\Big)={\mathcal
  Y}_d\Big(\frac{v}{\sqrt d}\Big)=0\right]$$
  $$\times p_{u,v}(0,0)J_{\phi}\Big(\frac{u}{\sqrt d}\Big)J_{\phi}\Big(\frac{v}{\sqrt d}\Big)dudv$$
$$
\mathop{\to}_{d\to\infty}
\int_{(\left[\frac12,\frac12\right]^m)^2}
\E[f(\mathcal Y'_\infty(u))f(\mathcal Y'_\infty(v))\mid\mathcal Y_{\infty}(u)=\mathcal Y_{\infty}(v)=0]
\; p_{{\cal Y}_\infty(u),{\cal Y}_\infty(v)}(0,0)dudv$$
$$=\E\left[\left(\mathcal V\left(\mathcal{C}_{\mathcal Y_\infty}(\mathbf{0})\cap \left[-\frac12,\frac12\right]^m\right)\right)^2\right]
<\infty.$$
The  term of the square of the expectation is easier. 

The same arguments show that for all $q$ we have 
\begin{equation*}
V^{loc}_{q,d}:=\Var\Big(\pi_q\Big(\frac{\mathcal V(\mathcal{C}_{\mathcal Y_d}(\mathbf{0})\cap \tilde R)}{d^{\frac r2-\frac m4}}\Big)\Big)
\mathop{\to}_{d\to\infty} 
\Var\left(\pi_q\left(\mathcal V\left((\mathcal{C}_{\mathcal Y_\infty}(\mathbf{0}))\cap \left[-\frac12,\frac12\right]^m\right)\right)\right)=:V^{loc}_{q}.
\end{equation*}
Thus, for all $Q$ it follows that  
$
\sum^Q_{q=0}V^{loc}_{q,d}
\mathop{\to}_{d\to\infty} \sum^Q_{q=0}
V^{loc}_{q}.
$
By Parseval's identity, \eqref{eq:conv-tan} can be written as
$$
\sum^{\infty}_{q=0}V^{loc}_{q,d}\mathop{\to}_{d\to\infty} \sum^{\infty}_{q=0}V^{loc}_{q}.
$$
Thus, by taking the difference we get 
\begin{equation}\label{eq:diff}
\sum_{q> Q}V^{loc}_{q,d}\mathop{\to}_{d\to\infty} \sum_{q>Q}V^{loc}_{q}. 
\end{equation}
Given that the series $\sum^\infty_{q=0}V^{loc}_{q}$ is convergent, 
we can choose $Q_0$ such that for $Q\geq Q_0$ it holds $\sum^\infty_{q>Q}V^{loc}_{q}\le \varepsilon/2$. 
Hence, for this $Q_0$ and by using \eqref{eq:diff} we can choose $d_0$ such that for all $d>d_0$ and $Q\geq Q_0$
$$\sum_{q>Q}V^{loc}_{d,q}\le \varepsilon.$$ 
Namely, there exists $d_0$ such that for $Q\geq Q_0$
$$
\sup_{d>d_0}\E\left[\left(\pi^Q\left(\frac{\mathcal V(\mathcal{C}_{\mathcal Y_d}(\mathbf{0})\cap \tilde R)}{d^{\frac r2-\frac m4}}
\right)\right)^2\right]<\varepsilon.
$$
\end{proof}

\section{Appendix} \label{app}
Letendre \cite{lt} and Letendre \& Puchol \cite{lt-pu} have studied the asymptotic variance 
of the volume of the zero set. 
We consider this problem using 
another method. 

Write the variance as
\[
\Var(\mathcal V_{\mathbf Y_d }(\mathbf 0))=\E[(\mathcal V_{\mathbf Y_d}(\mathbf 0)^2]-(\E[\mathcal V_{\mathbf Y_d}(\mathbf 0])^2.
\]

Let $f$ and $f_\gamma$ be defined as in \eqref{d:f} and \eqref{d:fgamma} respectively. 

We have already computed 
the second term, as in \eqref{eq:ricevol}-\eqref{e:fz}, 
and for the first one we apply the Rice formula for the second moment 
(\cite[Ch.6]{aw}).
\begin{align}\label{secondmoment}
\E[\mathcal V_{\mathbf Y_d}(\mathbf 0)^2]
&= \int_{S^m\times S^m} \E[f(\bY'_d(t))f(\bY'_d(s))\,|\bY( t)=\bY(s)=0]  
p_{\bY( t),\bY( s)}( 0,0)d tds \notag\\
&= d^{r}\int_{\mathbb S^m\times\mathbb S^m}
\E[f_{\sigma}( \zeta_1) f_{\sigma}( \zeta_2)]p_{\bY(t),\bY(s)}(0,0)dtds,
\end{align}
where $p_{\bY(t),\bY(s)}( 0,0)$ is the density of the vector 
$(\bY(t),\bY(s))$. 
By independence it holds 
\[
p_{\bY(t),\bY(s)}(0,0)=\prod_{\ell=1}^rp_{Y_\ell( t),Y_\ell(s)}(0,0)
=\frac1{(2\pi)^r(1-\langle t,s\rangle^{2d})^{\frac r2}}.
\]

As in \eqref{e:fbeta}, we have
\begin{eqnarray*}
f_\gamma(\mathbf y')=\sum_{\bbeta} f_{\bbeta}(\gamma)\tilde{\bf H}_{\bbeta}(\mathbf y'),
\end{eqnarray*}
\begin{remark}
Let us point out that for $r=m$ it holds $f(\mathbf y')=|\det \mathbf y'|$. 
Here $\mathbf y'$ is the $m\times m$ matrix 
whose columns are the vectors $(y'_{\ell1},\ldots, y'_{\ell,m})$ for $\ell=1,\ldots,m.$ Furthermore, 
by the homogeneity of the determinant we 
have $f_{\gamma}(\mathbf y')=\gamma f(\mathbf y')$. The lack of this fact  when $r<m$ implies that we need a different approach for obtaining the 
asymptotic variance.
\end{remark}

Thus, the expansion \eqref{e:fbeta} 
and the bi-dimensional Mehler's formula \cite[Th.10.7]{aw} give 
\begin{align} \label{e:H}
  d^{r}\E[f_{\sigma}( \zeta_1) f_{\sigma}( \zeta_2 )] &=d^{r}\sum_{\bbeta}
  \left[f_{\bbeta} \left(\sigma\left(\langle
  t,s\rangle\right)\right)\right]^2\bbeta! \left[\mathcal D\left(\langle 
  t,s\rangle\right)\right]^{(|\bbeta|-\sum_{\ell=1}^{r}\beta_{\ell1})}\left[\rho\left(\langle t,
  s\rangle\right)\right]^{\sum_{\ell=1}^{r}\bbeta_{\ell1}} \\
  &=d^{r}\mathcal H\left(\langle t, s\rangle\right).\nonumber 
\end{align}
Using \eqref{eq:intS} we obtain
\begin{eqnarray*} 
\E[\mathcal{V}_{\mathbf Y_d}(\mathbf 0)^2]&=&\kappa_{m}\kappa_{m-1}\int_0^\pi\mathcal 
H(\cos(\theta))\frac{d^{r}}{(2\pi)^r}\frac1{(1-\cos^{2d}(\theta))^{\frac r2}}\sin^{m-1}(\theta) d\theta\\
&=&\kappa_{m}\kappa_{m-1}\frac{d^{r-\frac12}}{(2\pi)^r}\int_0^{\sqrt
  d\pi}\mathcal H\left(\cos\left(\frac z{\sqrt d}\right)\right)
  \frac{\sin^{m-1}\left(\frac z{\sqrt d}\right) }
  {\left(1-\cos^{2d}\left(\frac z{\sqrt d}\right)\right)^{\frac r2}}dz.
\end{eqnarray*}
By Parseval equality we have 
\[
\sum_{\mathbf\nu}|f_{\bbeta}(\gamma)|^2\bbeta!
=\int_{\R^{r\times m}}|f_{\gamma}(\mathbf x)|^2\varphi_{r\times m}
(\mathbf x)d\mathbf x\le m^r(\gamma\vee 1)^{2r}
\E\left[\sup_{1\le\ell\le r,\,1\le j\le m}\left|\frac{Y_{\ell j}(\mathbf
e_0)}{\sqrt d}\right|^{2r}\right].
\]
Since $\sigma^2\big(\frac z{\sqrt d}\big)\le 1$ (Lemma \ref{l:bounds}), then 
\begin{eqnarray} \label{norma2}
  \sum_{\bbeta}\left|f_{\bbeta}\left(\sigma\left( \frac z{\sqrt d}\right)\right)\right|^2\bbeta! 
  < \mathbf{C}.
\end{eqnarray} 
Therefore, we can interchange 
the series with the integral obtaining
\begin{eqnarray}\label{serie}
\E[(\mathcal{V}_{\mathbf Y_d}(\mathbf 0)^2]
=\kappa_{m}\kappa_{m-1}\frac{d^{r-\frac12}}{(2\pi)^r}\sum_{\bbeta}\int_0^{\sqrt d\pi}
  \mathcal H_{\bbeta}\left(\cos\left(\frac  z{\sqrt d}\right)\right)
  \frac{\sin^{m-1}\left(\frac z{\sqrt d}\right)}
  {\left(1-\cos^{2d}\left(\frac z{\sqrt d}\right)\right)^{\frac r2}}dz,
\end{eqnarray}
where $\mathcal H$ is as in \eqref{e:H} 
and 
$\mathcal H_{\bbeta}= \left[f_{\bbeta} \left(\sigma\left(\langle
  t,s\rangle\right)\right)\right]^2\bbeta! \left[\mathcal D\left(\langle 
  t,s\rangle\right)\right]^{(|\bbeta|-\sum_{\ell=1}^{r}\beta_{\ell1})}\left[\rho\left(\langle t,
  s\rangle\right)\right]^{\sum_{\ell=1}^{r}\bbeta_{\ell1}}$.

Thus, using the above notation and 
normalizing  we have
\begin{multline*}
\E\left[\left(\frac{\mathcal V_{\mathbf Y_d}(\mathbf 0}{d^{\frac r2-\frac{m}4}}\right)^2\right]=  
\kappa_{m}\kappa_{m-1}\frac{d^{\frac{m-1}2}}{(2\pi)^r}\int_0^{\sqrt d\pi}
\sum_{|\bbeta|\ge1}\mathcal H_{\bbeta}\left(\cos(\frac{z}{\sqrt d})\right)\ 
\frac{\sin^{m-1}\left(\frac z{\sqrt d}\right)}
{\left(1-\cos^{2d}\left(\frac z{\sqrt d}\right)\right)^{\frac{r}2}} dz\\
+\kappa_{m}\kappa_{m-1}\frac{d^{\frac{m-1}2}}{(2\pi)^r}
\int_0^{\sqrt d\pi}
\Big|f_{\mathbf 0} \Big( \sigma \Big( \frac{z}{\sqrt d} \Big) \Big) \Big|^2
\frac{\sin^{m-1}\left(\frac z{\sqrt d}\right)}
{\left(1-\cos^{2d}\left(\frac z{\sqrt d}\right)\right)^{\frac{r}2}} dz.
\end{multline*}
We start with the terms $|\bbeta|\geq 1$. 
To apply the dominated convergence theorem we must look for a uniform bound. But let us begin with a remark.
\begin{remark}\label{simetria}
The symmetrization  argument used in step 3 of section 3.2 of \cite{aadl} gives that the integral over $[\sqrt d\frac{\pi}2,\sqrt d\pi]$ of each term 
in the series (\ref{serie}) is equal to the integral of same term on $[0,\sqrt d\frac{\pi}2]$ except for a multiplication by $(-1)^{(d-1)|\bbeta|}$. 
In this form the bound that is obtained for applying the dominated convergence  theorem in the latter interval serves also for the former.
\end{remark}

Using Lemma \ref{l:bounds} it holds that there exists a $d_0$ such that for $\frac z{\sqrt d}<\frac\pi2$ and $d>d_0$ it holds
\[
|\rho|\leq \mathbf{C}\ (1+z^2)^2\exp(-2\alpha z^2)
\mbox{ and } \mathcal D\le \exp\big(-2\alpha z^2\big).
\] 
By the Remark \ref{simetria} it is enough to study only 
the interval $[0,\sqrt d\frac\pi2]$. In this form we get
\begin{multline*}
  \Bigg|\sum_{|\bbeta|\ge1}
  \Big[ f_{\bbeta} \Big( \sigma \Big( \frac{z}{\sqrt d} \Big)\Big) \Big]^2
  \mathbf\bbeta! 
  \Big[ \mathcal D \Big( \frac z{\sqrt d}\Big) \Big]^{
  |\bbeta| -\sum_{\ell=1}^{r}\beta_{\ell1} }
  \Big[ \rho \Big( \frac{z}{\sqrt d} \Big)\Big]^{\sum_{\ell=1}^{r}\beta_{\ell1}}\Bigg| \\
  \leq 
  \mathbf{C}\ \sum_{|\bbeta|\ge1}
  \Bigg[f_{\bbeta} \Big( \sigma \Big( \frac z{\sqrt d} \Big)\Big)\Bigg]^2
  \bbeta! \,(1+z^2)^2\exp(-2\alpha z^2)\\
  \leq  \mathbf{C}\ (1+z^2)^2\exp(-2\alpha z^2).
\end{multline*}
Above we have used (\ref{norma2}).

It remains to consider the integral over the interval $[0,z_0]$. 
But now the integrand can be bounded by
\begin{equation} \label{e:cerca0}
\mathbf{C}\ \frac1{\left(1-\cos^{2d}\left(\frac z{\sqrt d}\right)\right)^{\frac r2}}d^{\frac{m-1}2}
\sin^{m-1}\left(\frac z{\sqrt d}\right)
\le \mathbf{C}\ \frac1{(1-\exp(-2\alpha z^2))^{\frac r2}} 
z^{m-1},
\end{equation}
and the function on the right hand side is integrable whenever $r<m.$ 

In this manner, applying the dominated convergence theorem we get
\begin{multline*}
\lim_{d\to\infty}\kappa_{m}\kappa_{m-1}\frac{d^{\frac{m-1}2}}{(2\pi)^r}
\int_0^{\sqrt d\pi}
\sum_{|\bbeta|\ge1}\mathcal H_{\bbeta}\left(\cos\left(\frac{z}{\sqrt d}\right)\right)
\frac{\sin^{m-1}\left(\frac z{\sqrt d}\right)}
{\left(1-\cos^{2d}\left(\frac z{\sqrt d}\right)\right)^{\frac{r}2}} dz\\
=\kappa_{m}\kappa_{m-1}\frac{1}{(2\pi)^r}\int_0^{\infty}
\sum_{|\bbeta|\ge1}\mathcal H_{\bbeta}(z)\ \frac{z^{m-1}}{(1-\exp(-z^2))^{\frac{r}2}} dz.
\end{multline*}

The exact expression for the function $\mathcal H_{\bbeta} (z)$ is obtained 
by using the results of the step 2 of \cite{aadl}. 
In fact
\begin{multline*}
\mathcal H_{\bbeta}(z)
=\left|f_{\bbeta} \Big( \Big( \frac{1-(1+z^2)\exp(-z^2)}{1-\exp(-z^2)} \Big)^{1/2} \Big) \right|^2
\bbeta!\, \left(\exp\left(-\frac{z^2}2\right)\right)^{|\bbeta|-\sum_{ \ell=1}^r\beta_{\ell1}}\\
\times 
\Big(\frac{(1 - z^2 - \exp(-z^2))  \exp(-z^2/2) }
   {1-(1+z^2)\exp(-z^2)}\Big)^{\sum_{\ell=1}^r\beta_{\ell1}}.
\end{multline*}

To end our proof it only remains to consider the zero term in the expansion. Consider first 
\[
\mathcal I_d=
\kappa_{m}\kappa_{m-1}\frac{d^{\frac{m-1}2}}{(2\pi)^r}\int_0^{\sqrt d\pi}
\Big[ \Big| f_0 \Big( \sigma \Big( \frac{z}{\sqrt d} \Big)\Big)\Big|^2
-|f_0(1)|^2\Big]
\frac{\sin^{m-1}\left(\frac z{\sqrt d}\right)}
{\left(1-\cos^{2d}\left(\frac z{\sqrt d}\right)\right)^{\frac{r}2}} dz.
\]
But
\begin{align*}
  \left|
  \Big| f_0 \Big( \sigma\Big( \frac{z}{\sqrt d}\Big)\Big) \Big|^2  -|f_0(1)|^2 
  \right| 
    &=   \mathbf{C} \left|\int_{\R^{r\times m}}\int_0^1
         \sum_{\ell=1}^r\frac{\partial f_{\xi\sigma+(1-\xi)}(\mathbf x)}{\partial x_{\ell1}}
         d\xi \Big( 1- \sigma\Big( \frac{z}{\sqrt d} \Big) \Big)\varphi_{r\times m}(\mathbf x)d\mathbf x\right|\\
   &\leq \mathbf{C}\ \mathbf 1_{\{z\le z_0\}}+\mathbf C\exp(-2\alpha z^2)\mathbf 1_{\{z>z_0\}}.
\end{align*}
Yielding by the dominated convergence theorem that
\[
\lim_{d\to\infty}\mathcal 
I_d=\frac{\kappa_{m}\kappa_{m-1}}{(2\pi)^r}\int_0^{\infty}\left[\left|
f_0\left(\Big(\frac{1-(1+z^2)\exp(-z^2)}{1-\exp(-z^2)}\Big)^{1/2}
\right)\right|^2-|f_0(1)|^2\right]
\frac{z^{m-1}}{(1-\exp(-z^2))^{\frac{r}2}} dz.
\]
Finally, we will consider the remaining term.
In first place let us point out that
\[
\kappa_{m}\kappa_{m-1}|f_0(1)|^2\frac{d^{\frac{m-1}2}}{(2\pi)^r}\int_0^{\sqrt
d\pi}\sin^{m-1}\left(\frac z{\sqrt d}\right) dz
=\frac1{d^{r-\frac{m}2}}(\E[\mathcal V(\mathcal{C}_{\mathbf Y_d}(\mathbf{0}))])^2.
\]
Then, 
substracting this term we get
\begin{equation}\label{integrant}
\mathcal J_d 
  =\kappa_{m}\kappa_{m-1}\frac{d^{\frac{m-1}2}}{(2\pi)^r}
     \int_0^{\sqrt d\pi}|f_0(1)|^2
      \Bigg[ \frac1{\big(1-\cos^{2d}\big(\frac z{\sqrt d}\big)\big)^{\frac{r}2}}-1\Bigg]
      \sin^{m-1}\left(\frac z{\sqrt d}\right) dz \notag.
\end{equation}
The convergence at $0$ follows from
\[      
 \frac1{\big(1-\cos^{2d}\big(\frac z{\sqrt d}\big)\big)^{\frac{r}2}}-1
 \leq \frac1{\big(1-\cos^{2d}\big(\frac z{\sqrt d}\big)\big)^{\frac{r}2}},
\]
using \eqref{e:cerca0}.

On the large interval we use 
the lower bound for $1-{\cal C}^2$ and the upper bound for ${\cal C}$ 
in Lemma \ref{l:bounds} to obtain 
\[
\left|d^{\frac{m-1}2}\left[\frac1{\left(1-\cos^{2d}\left(\frac z{\sqrt
d}\right)\right)^{\frac{r}2}}-1\right]\sin^{m-1}\left(\frac z{\sqrt
d}\right)\right|
\le \mathbf{C}\ \exp(-2\alpha 
z^2)\frac{z^{m-1}}{(1-\exp(-2\alpha z^2))^{\frac{3r}2}}.
\]
Since these two bounds allow applying the dominated convergence theorem it holds
\[
\lim_{d\to\infty}\mathcal 
J_d=\kappa_{m}\kappa_{m-1}\frac{|f_0(1)|^2}{(2\pi)^r}
\int_0^{\infty}\left[
\frac1{\left(1-\exp(-z^2)\right)^{\frac{r}2}}-1
\right]z^{m-1}dz.
\]
Hence, it is possible to write a closed formula for the limit variance. 
Indeed, we obtain
\begin{align*}
\lim_{d\to0}\Var\left(\frac{\mathcal V_{\mathbf
  Y_d}(\mathbf 0)}{d^{\frac r2-\frac{m-1}4}}\right)&=
  \kappa_{m}\kappa_{m-1}\left\{\frac{1}{(2\pi)^r}
  \int_0^{\infty}\left(\sum_{|\bbeta|\ge1}\mathcal 
H_{\bbeta}(z)\right)\frac1{(1-\exp(-z^2))^{\frac{r}2}} z^{m-1}dz\right.\\
  &
+\frac{1}{(2\pi)^r}\int_0^{\infty}\left[\left|f_0\left(\Big(\frac{1-(1+z^2)\exp(-z^2)}
{1-\exp(-z^2)}\Big)^{1/2}\right)\right|^2-|f^r_0(1)|^2\right]
\frac{z^{m-1}}{(1-\exp(-z^2))^{\frac{r}2}} dz\\
  &\left.+
\frac{|f_0(1)|^2}{(2\pi)^r}
  \int_0^{\infty}\left[\frac1{\left(1-\exp(-z^2)\right)^{\frac{r}2}}-1\right]z^{m-1}dz\right\}
\end{align*}
The result follows.
\qed

\bibliographystyle{imsart-nameyear}

\begin{thebibliography}{}

\bibitem{arcones}
M. Arcones. 
Limit theorems for nonlinear functionals of a stationary Gaussian sequence of vectors. 
{\em Ann. Probab.}, 22, no. 4, 2242-2274 (1994). 

\bibitem{aadl}
D. Armentano, J-M. Aza\"is,  F. Dalmao and J.R. Le\'{o}n. 
On the asymptotic variance of the number of real roots of random polynomial systems. 
{\em Proc. Amer. Math. Soc.}, 147, no. 1, 205-214 (2019).

\bibitem{aadl2}
D. Armentano, J-M. Aza\"is,  F. Dalmao and J.R. Le\'{o}n. 
Central Limit Theorem for the number of real roots of Kostlan--Shub--Smale
random polynomial systems. To appear in {\em American Journal of
Mathematics}.

\bibitem{al}
J.M. Aza\"is, J.R. Le\'on. 
Necessary and sufficient conditions for the finiteness of the second moment of the measure of level sets. 
{\em Electron. J. Probab.}, 25 (2020), Paper No. 107, 15 pp.

\bibitem{aw}
J-M Aza\"{i}s and M. Wschebor. 
Level sets and extrema of random processes and fields. 
John Wiley \& Sons Inc., Hoboken, NJ, 
ISBN: 978-0-470-40933-6 (2009).

\bibitem{ber} 
C. Berzin. 
Estimation of Local Anisotropy Based on Level Sets. arXiv:1801.03760 




\bibitem{Kac}
M. Kac.
On the average number of real roots of a random algebraic equation. 
{\em Bull. Amer. Math. Soc.}, 49 , 314-320 (1943).

\bibitem{Kos}
E. Kostlan. 
On the distribution of roots of random polynomials. In The work of Smale
in differential topology, From Topology to Computation: Proceedings of
the Smalefest, (Berkeley, CA, 1990), 419-431, Springer, New York,
(1993).

\bibitem{lt}
T. Letendre.
Variance of the volume of random real algebraic submanifolds. 
{\em Trans. Amer. Math. Soc.}, 371, no. 6, 4129-4192 (2019).

\bibitem{lt-pu}
T. Letendre T, M. Puchol. 
Variance of the volume of random real algebraic submanifolds II,
{\em Indiana University Mathematics Journa}l, 68 (6), pp.1649-1720 (2019).

 \bibitem{Mas}
N. B. Maslova.
The variance of the number of real roots of random polynomials.
{\em Teor. Verojatnost. i Primenen}, 19, 36-51 (1974).

\bibitem{Mas2}
N.B. Maslova. The distribution of the number of real roots of random polynomials. 
{\em Teor. Verojatnost. i Primenen}, 19, 488-500 (1974).

\bibitem{np}
I. Nourdin  and G. Peccati. Normal approximations with Malliavin calculus. 
From Stein's method to universality. Cambridge Tracts in Mathematics, 192. Cambridge University Press, Cambridge, xiv+239 pp. ISBN: 978-1-107-01777-1 (2012). 

\bibitem{ns}
F. Nazarov, M. Sodin. 
Asymptotic laws for the spatial distribution and the number of connected components of zero sets of Gaussian random functions. 
{\em Journal of Mathematical Physics, Analysis, Geometry}, V12, No. 3 pp 205-278 (2016).

\bibitem{pta}
G. Peccati and M. Taqqu. 
Wiener chaos: moments, cumulants and diagrams. A survey with computer implementation. 
Supplementary material available online. Bocconi \& Springer Series, 1. 
Springer, Milan; Bocconi University Press, Milan, xiv+274 pp. ISBN: 978-88-470-1678-1 (2011).

\bibitem{bez1}
  M.~Shub and S.~Smale.
   Complexity of {B}\'ezout's {T}heorem {I}: geometric aspects.
   {\em Journal of AMS}, 6:459--501, (1993).
\bibitem{ss}
M. Shub and S. Smale. 
Complexity of B\'ezout's theorem. II. Volumes and 
Computational algebraic geometry (Nice, 1992), 267-285, 
Progr. Math., 109, Birkh\"auser Boston, Boston, MA, (1993). 

\bibitem{bez3}
  M.~Shub and S.~Smale.
   Complexity of {B}\'ezout's Theorem III: condition 
  number and packing.
   {\em Journal of Complexity}, 9:4--14, (1993).

\bibitem{bez4}
  M.~Shub and S.~Smale.
   Complexity of {B}\'ezout's {T}heorem {IV}: probability of 
  success; extensions.
   {\em SIAM J. of Numer. Anal.}, 33:128--148, (1996).

 \bibitem{bez5}
M.~Shub and S.~Smale.
 Complexity of {B}\'ezout's {T}heorem {V}: polynomial time.
 {\em Theoretical Computer Science}, 133:141--164, (1994).


\bibitem{ws-var}
M. Wschebor 
On the Kostlan-Shub-Smale model for random polynomial systems. Variance of the number of roots. 
{\em J. Complexity}, 21 No. 6, 773-789 (2005).
\end{thebibliography}

\end{document}